\documentclass{amsart}
\usepackage{amsmath, amssymb, latexsym, multicol, color, enumerate,comment, amsthm}
\usepackage{lscape}

\newtheorem{theorem}{Theorem}[section] % 1st argument is your name for it
\newtheorem{lemma}[theorem]{Lemma}     % 2nd argument is what is printed
\newtheorem{corollary}[theorem]{Corollary}

\newcommand{\Z}{\mathbb{Z}}
\newcommand{\Hy}{\mathbb{H}}
\newcommand{\SR}{\mathcal S\cup\mathcal R}
\newcommand{\btu}{\bigtriangleup}
\newcommand{\lan}{\langle}
\newcommand{\ran}{\rangle}
\newcommand{\ord}{\text{ord}}

\title[spinor regular ternary quadratic forms]% end with percent
 {Completeness of the list of spinor regular ternary quadratic forms} % This is the full title of the paper
\author{A. G. Earnest}
\address{A.G. Earnest, Department of Mathematics, Southern Illinois University, Carbondale, IL 62901, USA}
\email{haenscha@duq.edu}

\author{Anna Haensch}
\address{Anna Haensch, Department of Mathematics and Computer Science, Duquesne University, Pittsburgh, PA 15282, USA.}
\email{haenscha@duq.edu}

\dedicatory{Dedicated to the memory of Professor Irving Kaplansky.}


\thanks{The research of the second author was supported by an Association for Women in Mathematics Mentoring Travel Grant.}

\begin{document}
\maketitle

\begin{abstract}
Extending the notion of regularity introduced by Dickson in 1939, a positive definite ternary integral quadratic form is said to be spinor regular if it represents all the positive integers represented by its spinor genus (that is, all positive integers represented by any form in its spinor genus). Jagy conducted an extensive computer search for primitive ternary quadratic forms that are spinor regular, but not regular, resulting in a list of 29 such forms. In this paper, we will prove that there are no additional forms with this property.
\end{abstract}
 
\section{Introduction} 
 
The objects of study in this paper will be positive definite ternary primitive integral quadratic forms, which for convenience will be referred to simply as ternary quadratic forms, or ternaries, throughout the paper. The systematic study of spinor regular ternary quadratic forms was initiated by Benham, Earnest, Hsia and Hung in 1990 \cite{BEHH}. In that paper, the authors prove that there exist only finitely many isometry classes of forms with this property, and determine all of them for which the discriminant does not exceed 1000 and the genus splits into multiple spinor genera, producing a list of eleven forms that are spinor regular, but not regular, in that range. Following the completion of the determination of all possible regular ternary quadratic forms by Jagy, Kaplansky and Schiemann \cite{JKS} (the regularity of all forms on their list has now been essentially confirmed; the reader is referred to papers of Lemke Oliver \cite{LO} and Oh \cite{Oh} for recent developments), Kaplansky suggested to the first author of the present paper in 1997 that ``The time may be ripe to find all spinor regular (positive ternary) forms". The present paper brings that vision to fruition, providing the final step in the completion of that project.

Subsequent to the work in \cite{BEHH}, the search for additional spinor regular ternaries with discriminants larger than 1000 was pushed forward by Jagy, using an extensive computer search that ultimately produced a list of 29 such forms, the largest discriminant of which is $87,808=2^8\cdot 7^3$. This list was posted to the Catalogue of Lattices in May 2004 and is reproduced here in Table 1 for ease of reference. As of the time of posting, Jagy's exhaustive search had determined that no further spinor regular ternaries exist with discriminant less than 575,000. The remainder of this paper will be devoted to a proof that the list in Table 1 is in fact complete.

\begin{table}[!htb]
\begin{tabular}{l|c|c}
\hline
\text{$disc(f)$} & \text{Coefficients of $f$} & \text{Size of Spinor Genus}\\
\hline
 $64=2^6 $& [2, 2, 5, 2, 2, 0] & 1\\
       $108=2^2 \cdot 3^3 $& [3, 3, 4, 0, 0, 3] & 1\\
       $108= 2^2 \cdot 3^3$& [3, 4, 4, 4, 3, 3] & 1\\
       $128=2^7 $& [1, 4, 9, 4, 0, 0] & 1\\
       $256=2^8$& [2, 5, 8, 4, 0, 2] & 1\\
       $256=2^8$& [4, 4, 5, 0, 4, 0] & 1\\
       $324=2^2\cdot 3^4$& [1, 7, 12, 0, 0, 1] & 1\\
       $343=7^3$& [2, 7, 8, 7, 1, 0] & 1\\
       $432=2^4\cdot 3^3$& [3, 7, 7, 5, 3, 3]  & 2\\
       $432=2^4\cdot 3^3$& [4, 4, 9, 0, 0, 4] & 1\\
       $432=2^4\cdot 3^3$& [3, 4, 9, 0, 0, 0] & 1\\
       $1024=2^{10}$& [4, 9, 9, 2, 4, 4] & 1\\
       $1024=2^{10}$& [4, 5, 13, 2, 0, 0] & 1\\
       $1024=2^{10}$& [5, 8, 8, 0, 4, 4] & 1\\
       $1372=2^2\cdot 7^3$& [7, 8, 9, 6, 7, 0] & 1\\
       $1728=2^6\cdot 3^3$& [4, 9, 12, 0, 0, 0] & 1\\
       $2048=2^{11}$& [4, 8, 17, 0, 4, 0] & 1\\
       $3888=2^4\cdot 3^5$& [4, 9, 28, 0, 4, 0] & 1\\
       $4096=2^{12}$& [9, 9, 16, 8, 8, 2] & 1\\
       $4096=2^{12}$& [4, 9, 32, 0, 0, 4] & 1\\
       $4096=2^{12}$& [5, 13, 16, 0, 0, 2] & 1\\
       $5488=2^4\cdot 7^3$& [8, 9, 25, 2, 4, 8] & 1\\
       $6912=2^8\cdot 3^3$& [9, 16, 16, 16, 0, 0] & 1\\
       $6912=2^8\cdot 3^3$& [13, 13, 16, -8, 8, 10] & 1\\
       $16384=2^{14}$& [9, 17, 32, -8, 8, 6] & 1\\
       $16384=2^{14}$& [9, 16, 36, 16, 4, 8] & 1\\
       $27648=2^{10}\cdot 3^3$& [9, 16, 48, 0, 0, 0]  & 2\\
       $62208=2^8\cdot 3^5$& [9, 16, 112, 16, 0, 0] & 1\\
       $87808=2^8\cdot 7^3$& [29, 32, 36, 32, 12, 24] & 1\\
       \hline
\end{tabular}
\caption{Jagy's table of forms which are spinor regular but not regular. The sextuple $[a,b,c,d,e,g]$ corresponds to the ternary form $f(x,y,z)=ax^2+by^2+cz^2+dyz+exz+gxy$.}\label{Tab:1}
\end{table}

\begin{theorem}\label{main} Every spinor regular ternary quadratic form that is not regular appears in Table 1. \end{theorem}

In particular, a ternary quadratic form that lies in a spinor genus consisting of a single equivalence class is spinor regular. Consequently, we have the following immediate corollary. 

\begin{corollary} Every ternary quadratic form that lies in a one-class spinor genus but not a one-class genus appears in Table 1.\end{corollary}

Within Table 1, 27 of the 29 forms appearing lie in one-class spinor genera. Proofs of the spinor regularity of the remaining two forms can be found in \cite{BEHH}. The strategy to be used in the proof of the above theorem originates in a paper of Chan and Earnest \cite{CE}. While the approach of \cite{BEHH} relies on asymptotic methods and is ineffective for producing a bound for the largest possible discriminant of a spinor regular ternary quadratic form, the methods introduced in \cite{CE} produced 1) a list of the possible prime divisors that can occur in the factorization of the discriminant of a spinor regular ternary form, and 2) for each prime $p$ in that list, an upper bound for the largest power of $p$ that can occur in such a discriminant. Consequently, there is a specific finite list of all possible discriminants for spinor regular ternaries. However, this list contains many discriminants that are well beyond the bound of Jagy's systematic search. For example, for the primes 2 and 3, the results in \cite{CE} only restrict the powers of those primes dividing potential discriminants to not exceed $2^{28}$ and $3^{13}$, respectively. Our approach here will be to refine and extend the techniques of \cite{CE} in order to perform a targeted sequential elimination of the remaining discriminants that lie outside the search bound. For this purpose, we will make extensive use of the explicit knowledge of the lattices corresponding to the 913 candidates for regular ternaries given in \cite{JKS} and the known spinor regular forms given in Table 1, and particularly of their $p$-adic local structures.  Local structures will be computed using SageMath \cite{sage}, in particular, its functionality for quadratic forms and integral lattices. 

The paper will be organized as follows.  In Section \ref{prelim}, we will introduce the necessary terminology and notation along with several technical lemmas, incorporating the relevant results from the previous paper \cite{CE}.  In Section \ref{results}, we will give a sequence of lemmas systematically eliminating potential discriminants, culminating in a proof of Theorem \ref{main}.  Auxiliary tables in support of the arguments in Section \ref{results} are found in the Appendix. 

\section{Notation and Preliminary Lemmas}\label{prelim}

Throughout the remainder of this paper, we will adopt the geometric language of quadratic lattices, following the terminology and notation from O'Meara's book \cite{OM} and the previous paper \cite{CE}. Let us first recall the conventions from \cite{CE} regarding the correspondence between quadratic forms and lattices. We use the sextuple $[a,b,c,d,e,g]$ to denote the form 
\[
f(x,y,z)=ax^2+by^2+cz^2+dyz+exz+gxy
\]
where $a,b,c,d,e,g\in \Z$, and associate to $f$ the symmetric matrix of second partial derivatives
\[
F=\begin{bmatrix}
2a  & g & e \\
g & 2b & d\\
e & d & 2c
\end{bmatrix}.
\]
The discriminant of $f$ is then given by 
\[disc(f) = \frac{1}{2}det(F),\]
where det denotes the determinant. The lattice associated to such a form $f$ is the quadratic $\Z$-lattice $L$ with Gram matrix $F$ with respect to some basis. Then $L$ has norm ideal $\frak n (L)\subseteq 2\Z$ and scale ideal $\frak s (L)\subseteq \Z$. Moreover, the form $f$ is primitive precisely when $\frak n (L)=2\Z$ (such lattices are referred to as {\em normalized} in \cite{CE}). Note that if $dL$ denotes the discriminant of the lattice $L$,  then 
\[
dL=2disc(f).
\]
Define $\mathcal{R}$ and $\mathcal{S}$ to be the sets consisting of the discriminants of regular ternary quadratic forms, and spinor regular forms that are not regular, respectively. So all discriminants in $\mathcal{R}$ are known by \cite{JKS}, and the problem at hand is to prove that the set $\mathcal{S}$ consists of only the discriminants appearing in Table 1.

Throughout the remainder of the paper, the term {\em ternary lattice} will always refer to a positive definite ternary quadratic $\Z$-lattice with $\frak n (L)=2\Z$, and {\em spinor regular lattice} will refer to a ternary lattice that corresponds to a spinor regular ternary quadratic form.  For a prime $p$, $\Z_p$ will denote the ring of $p$-adic integers and $\Z_p^\times$ the group of units in $\Z_p$.  For a ternary lattice $L$, $L_p$ will denote the quadratic $\Z_p$-lattice $L\otimes_\Z\Z_p$. 

The main feature of \cite{CE} that will be exploited here is the introduction of a family of lattice transformations analogous to transformations on quadratic forms first used by Watson in his systematic study of regular ternary quadratic forms in the unpublished thesis \cite{W}. These lattice transformations, which are now commonly referred to as Watson transformations, have become ubiquitous in the literature on quadratic forms satisfying various regularity conditions. Several variations of these transformations have been developed, depending on which convention regarding the scaling of the lattices at the prime 2 is adopted. For ease of reference to results from the paper \cite{CE}, we will utilize the version $\lambda_p$ of the transformations defined in \S 2 of that paper. For our purposes, the main property of these transformations is that, for lattices with discriminant divisible by $2p^2$, $\lambda_p$ preserves spinor regularity whenever the lattice is not $p$-adically split by a hyperbolic plane $\Hy$, while reducing the power of the prime $p$ dividing the discriminant of the resulting lattice and leaving the powers of the other prime factors unchanged.  

For the definition and basic properties of the $\lambda_p$ transformations the reader is referred to the previous paper \cite{CE}.  We will now summarize and extend some results from \cite{CE} that provide the details of the action of $\lambda_p$ on $L_p$ that will be needed throughout the remainder of the paper.

\begin{lemma}\label{descent}
Let $L$ be a spinor regular lattice and $p$ an odd prime such that $p^2\mid dL$ and 
\[
L_p\cong \lan a,p^\beta b,p^\gamma c\ran
\]
with $a,b,c\in \Z_p^\times$ is not split by $\mathbb H$.  Then $\lambda_p(L)$ is spinor regular and 
\[
(\lambda_p(L))_p\cong\begin{cases}
\lan a,b,p^{\gamma-2}c\ran & \text{ when }\beta=0\text{ and }\gamma\geq 2\\
\lan b,pa,p^{\gamma-1}c\ran & \text{ when }\beta=1\\
\lan a,p^{\beta-2}b,p^{\gamma-2}c\ran & \text{ when }\beta\geq 2.
\end{cases}
\] 
\end{lemma} 

\begin{proof}
This is just \cite[Lemma 2.7]{CE} and \cite[Proposition 3.2]{CE}.
\end{proof}
When $p=2$ we need to consider more possibilities for the local structure of $L_2$.  In cases where $L_2$ is not split by $\mathbb H$ or $\mathbb A$ (the latter is just the binary lattice isometric to the lattice $A(2,2)$ in O'Meara's notation, see \cite[$\S$ 106A]{OM}), $L_2$ must be one of
\begin{eqnarray}\label{diagonal}
\lan 2a,2^\beta b,2^\gamma c\ran
\end{eqnarray}
with $a,b,c\in \Z_2^\times$ and $\beta\leq \gamma$, 
\begin{eqnarray}\label{hyperbolic}
\lan 2a\ran \perp 2^\gamma \mathbb H,
\end{eqnarray}
or 
\begin{eqnarray}\label{anisotropic}
\lan 2a\ran \perp 2^\gamma \mathbb A.
\end{eqnarray}
These cases will suffice for the purposes of this paper. 

\begin{lemma}\label{descent2}
Let $L$ be a spinor regular lattice such that $2^4\mid dL$ and $L_2$ is not split by $\mathbb H$ or $\mathbb A$.  Then $\lambda_2(L)$ is spinor regular and 
\[
(\lambda_2(L))_2\cong \begin{cases}
\begin{bmatrix}
a+b & a-b\\
a-b & a+b
\end{bmatrix}\perp\lan2^{\gamma-1} c\ran & \text{ when $L_2$ has the form (\ref{diagonal}), } \beta=1\text{ and }a\equiv b\mod 4\\
&\\
\begin{bmatrix}
\frac{a+b}{2} & \frac{a-b}{2}\\
\frac{a-b}{2} & \frac{a+b}{2}
\end{bmatrix}\perp\lan2^{\gamma-2} c\ran & \text{ when $L_2$ has the form (\ref{diagonal}), }\beta=1\text{ and }a\not \equiv b\mod 4\\
&\\
\lan 2b,4a,2^{\gamma-1} c\ran& \text{ when $L_2$ has the form (\ref{diagonal}) and }\beta=2\\
\lan 2a,2^{\beta-2} b,2^{\gamma-2} c\ran& \text{ when $L_2$ has the form (\ref{diagonal}) and }2< \beta\leq \gamma\\
\lan 2a\ran \perp 2^{\gamma-2} \mathbb H& \text{ when $L_2$ has the form (\ref{hyperbolic})}\\
\lan 2a\ran \perp 2^{\gamma-2} \mathbb A& \text{ when $L_2$ has the form (\ref{anisotropic})}.
\end{cases}
\]
\end{lemma}

\begin{proof}
The fact that $\lambda_2(L)$ is spinor regular under these assumptions is just \cite[Proposition 3.2]{CE}.  Suppose we have a spinor regular ternary lattice, $L$, with $\ord_2(dL)>3$ which is not split at 2 by $\mathbb H$ or $\mathbb A$.  Since in every case $\mathfrak{s}(L)\subseteq 2\Z_2$, by \cite[Lemma 2.2(g)]{CE} we have $\Lambda_4(L_2)=\{x\in L_2:Q(x)\in 4\Z_2\}$, and hence
\[
\Lambda_4(L_2) =\begin{cases}
\begin{bmatrix}
2a+2b & 2a-2b\\
2a-2b & 2a+2b
\end{bmatrix}\perp\lan2^\gamma c\ran & \text{ when $L_2$ has the form (\ref{diagonal}) and }\beta=1 \\
\lan 8a, 4b, 2^\gamma c\ran & \text{ when $L_2$ has the form (\ref{diagonal}) and }\beta=2\\
\lan 8a, 2^\beta b, 2^\gamma c\ran & \text{ when $L_2$ has the form (\ref{diagonal}) and }\beta>2\\
\lan 8a\ran \perp 2^\gamma\mathbb H& \text{ when $L_2$ has the form (\ref{hyperbolic})}\\
\lan 8a\ran \perp 2^\gamma\mathbb A & \text{ when $L_2$ has the form (\ref{anisotropic})},
\end{cases}
\]
but now the result follows immediately from the definition  
\[
\lambda_2(L)=\begin{cases}
\Lambda_{4}(L)^{1/2} & \text{ if }\mathfrak{n}(\Lambda_{4}(L))=4\Z\\
\Lambda_{4}(L)^{1/4} & \text{ if }\mathfrak{n}(\Lambda_{4}(L))=8\Z
\end{cases}
\]
coming from \cite{CE}.
\end{proof}

For a prime $q\neq p$, $\lambda_p(L_q)=(\lambda_p(L))_q$ is isometric to the scaling of $L_q$ by an element of $\Z_q^\times$.  Consequently the powers of $q$ in the discriminants $dL$ and $d(\lambda_p(L))$ are identical for $q\neq p$; hence, $dL$ and $d(\lambda_p(L))$ differ only by a power of $p$. 

\begin{lemma}\label{order}
Let $L$ be a spinor regular ternary lattice such that $p\mid dL$ and $L_p\cong \lan a,p^\beta b,p^\gamma c\ran $ with $a,b,c\in \Z_p^\times$ and $0\leq \beta\leq \gamma$. Then
\begin{enumerate}
\item If $\beta=0$ then
\[
\gamma\leq \begin{cases}
1 & \text{ if }p=7,11,13\\
2 & \text{ if }p=5\\
4& \text{ if }p=3.
\end{cases}
\]
\item If $\beta=1$ then 
\[
\gamma\leq \begin{cases}
1 & \text{ if }p=13\\
2 & \text{ if }p=5,7,11\\
4& \text{ if }p=3.
\end{cases}
\]
\item If $\beta\geq 2$, then $p=3$ or 7 and 
\[
\begin{cases}
\beta=2\text{ and }\gamma=3 & \text{ for }p=7,\\
\beta\leq 5\text{ and }\gamma\leq 8 & \text{ for }p=3.
\end{cases}
\]
\end{enumerate}
\end{lemma}

\begin{proof}
This is just \cite[Proposition 5.4]{CE}.
\end{proof}

We have omitted the information for $p=17,23$ presented in the original proposition, but as we will see shortly, these values will become obsolete. For an odd prime $p$, let $r_p$ denote the largest power of the prime $p$ occurring in the factorization of any discriminant in $\mathcal{R}$. These values can be obtained by an examination of the factorizations of the discriminants of the forms appearing in the list in \cite{JKS}. When $p=2$, we fix $r_2=8$, since from \cite[Proposition 5.5(a)]{CE} we know that a lattice can only be split by $\mathbb H$ if the order of 2 in its discriminant is less than 8.  For our purposes, the following values of $r_p$ will be relevant,
\[
r_2=8, r_3=5, r_5=3, r_7=2, r_{11}=2, r_{13}=2.
\]
Then we have the following useful corollary. 

\begin{corollary}\label{notsplit}
For a spinor regular ternary lattice $L$, if $\ord_p(dL)\geq r_p$, then $L_p$ is not split by $\mathbb H$. 
\end{corollary}

Now we can begin to use the lemmas above together with results from the previous paper, \cite{CE}, to begin eliminating certain factors from the elements in $\mathcal S$. Recall from \cite{CE}, that a lattice $L$ is said to {\em behave well at a prime $p$} if either $2p^2$ does not divide $dL$ or $L_p$ is split by $\mathbb H$, and $L$ is said to {\em behave well away from $p$} if $L$ behaves well at $q$ for all primes $q\neq p$.

\begin{lemma}\label{primepair}
If the product of two distinct odd primes $p$ and $q$ divides a discriminant in $\mathcal S$, then $p\cdot q\in \{3\cdot 5,\,\,\,3\cdot 7,\,\,\,3\cdot 11,\,\,\,3 \cdot 13,\,\,\, 5\cdot 7,\,\,\,11\cdot 13\}$.
\end{lemma}

\begin{proof}
Suppose that $p\cdot q$ divides a discriminant in $\mathcal S$ for distinct odd primes $p$ and $q$.  Then there is some spinor regular lattice $L$ with $p\cdot q\mid dL$.  Then by \cite[Lemma 4.1]{CE} there exists a spinor regular ternary lattice $L'$ that behaves well away from $p$ and $q$ and $\ord_p(dL)=\ord_p(dL')$ and $\ord_q(dL)=\ord_q(dL')$.  From \cite[Lemma 3.5]{CE} it follows that $\Z_r^\times\subseteq \theta(O^+(L'_r))$ for any prime, $r$, not equal to $p$ or $q$. 

Locally at $p$ we have $L'_p\cong \lan a,p^\beta b,p^\gamma c\ran$ where $0\leq \beta\leq \gamma$ and $a,b,c\in \Z_p^\times$.  If $\beta=0$ or if $\beta=\gamma$ then $\Z_p^\times\subseteq \theta(O^+(L'_p))$.  Suppose that $0<\beta<\gamma$.  If $\beta=2\beta_0$ is even, then
\[
(\lambda_p^{\beta_0}(L'))_p=\lan a,b,p^{\gamma'}c\ran,
\]
where $\gamma'=\gamma-2\beta_0>0$.  If $\beta=2\beta_0+1$ is odd, then 
\[
(\lambda_p^{\beta_0+\gamma'}(L'))_p=\lan a',pb',c\ran,
\]
where $a'=a, b'=b$ if $\gamma'$ is even, and $a'=b, b'=a$ is $\gamma'$ is odd.  Then setting 
\[
\tilde{L}=\begin{cases}
L' & \text{ if }\beta=0 \text{ or }\beta=\gamma,\\
\lambda_p^{\beta_0}(L') & \text{ if }\beta>0\text{ is even}\\
\lambda_p^{\beta_0+\gamma'}(L') & \text{ if }\beta>0 \text{ is odd},
\end{cases}
\]
we have $p\mid d\tilde{L}$ and $\Z_p^\times\subseteq \theta(O^+(\tilde{L}_p))$ for any choice of $\beta$. Since $\tilde{L}_r$ is just a scaling of $L_r'$ by a unit in $\Z_r^\times$ for primes $r\neq p$, we have $\Z_r^\times\subseteq \theta(O^+(\tilde{L}_r))$ for $r\neq q$ and $\ord_q(d\tilde{L})=\ord_q(dL')$. 

Now we can apply the argument above a second time, by repeatedly applying $\lambda_q$ to $\tilde{L}$ to arrive at $\hat{L}$ for which $q\mid d\hat{L}$ and $\Z_q^\times\subseteq \theta(O^+(\hat{L}_q))$. But now we have $\Z_r^\times\subseteq \theta(O^+(\hat{L}_r))$ at every prime $r$, and it follows from \cite[102:9]{OM} that $\hat{L}$ is regular.  However the only products of distinct odd primes appearing in $\mathcal R$ are in the set
\[
\{3\cdot 5,\,\,\,3\cdot 7,\,\,\,3\cdot 11,\,\,\,3 \cdot 13,\,\,\, 5\cdot 7,\,\,\,11\cdot 13\},
\]
thus proving the claim. 
\end{proof}

\begin{lemma}\label{triples}
Any discriminant in $\mathcal S$ has at most 2 odd prime factors. 
\end{lemma}

\begin{proof}
Suppose there is a discriminant in $\mathcal{S}$ which has distinct odd prime factors $p,q$ and $r$.  Then from Lemma \ref{primepair} it must follow that $(p,q,r)\in \{(3,5,7),(3,11,13)\}$.  Then there is some spinor regular lattice $L$ which behaves well away from $p,q$ and $r$ and $p\cdot q\cdot r\mid dL$.  Then we can repeatedly apply the argument from Lemma \ref{primepair} to obtain a regular lattice $K$ with $p\cdot q\cdot r\mid dK$, but for either triple, no such regular lattice exists. 
\end{proof}

\begin{lemma}\label{primefactor}
If $p$ is an odd prime that divides a discriminant in $\mathcal S$ then $p\in \{3,5,7,11,13\}$. 
\end{lemma}

\begin{proof}
It is already shown in \cite[Corollary 4.4]{CE} that the only primes dividing a discriminant in $\mathcal S$ are those in the set $\{2,3,5,7,11,13,17,23\}$.  Here we will show that $17$ and $23$ also cannot divide a discriminant in $\mathcal S$.  

Suppose $L$ is a spinor regular (but not regular) lattice and $17\mid dL$.  It follows from Lemma \ref{primepair} that $dL$ does not have any other odd prime divisors, so $dL=2^k\cdot 17^m$ and $L$ behaves well away from 2 and 17. From \cite[Proposition 5.4]{CE} we know that $L_{17}$ has the local structure $\lan a,b,17c\ran$ or $\lan a,17b,17c\ran$ where $a,b,c\in \Z_{17}^\times$, and consequently $\Z_{17}^\times\subseteq \theta(O^+(L_{17}))$.  Therefore we may assume that $8\mid dL$ and $L_2$ is not split by $\mathbb H$, or else $L$ would be regular.  If $3\leq k\leq 10$, then we are done since $dL\leq 2^{10}\cdot 17^2<575,000$ but no discriminant of the form $2^k\cdot 17^m$ appears in Table 1.  If $k>3$, then it follows from \cite[Proposition 3.2]{CE} that $\lambda_2(L)$ is spinor regular with $k-4\leq \ord_2(d(\lambda_2(L)))\leq k-1$.  Since $L_2$ is not split by $\mathbb H$ then $(\lambda_2(L))_2$ is also not split by $\mathbb H$, so we can continue to descend by $\lambda_2$ until we reach a discriminant of the form $2^k\cdot 17^m$ with $7\leq k\leq 10$ and $m=1,2$.  At this point we are within our search bound of 575,000, but since such a $2^k\cdot 17^m$ is not in $\SR$, we have reached a contradiction.  Therefore, 17 does not divide any discriminant in $\mathcal S$. 

Next, suppose that $L$ is a spinor regular (but not regular) lattice and $23\mid dL$.  Again it follows from Lemma \ref{primepair} that $dL$ does not have any other odd prime divisors, so $dL=2^k\cdot 23^m$, and again it follows from \cite[Lemma 2.5]{CE}\ and \cite[Proposition 5.4]{CE} that $L_{23}$ has the local structure $\lan a,b,23c\ran$ or $\lan a,23b,23c\ran$ where $a,b,c\in \Z_{23}^\times$.  From here the argument proceeds precisely as above.
\end{proof}

\begin{lemma}\label{odd}
Suppose $p$ is an odd prime and $d_0$ is an integer with $\gcd(p,d_0)=1$.  If there exists an integer $t>r_p-3$ such that $p^{t}d_0,\,\,p^{t+1}d_0,\,\,p^{t+2}d_0$ and $p^{t+3}d_0$ are all not in $\mathcal{S}$ or $\mathcal{R}$, then $p^{t_0}d_0\not\in \mathcal{S}$ for all $t_0\geq t$.  
\end{lemma}

\begin{proof}
If the result is not true, then there exists some minimal integer $t_0>t>r_p-3$ such that $p^{t_0}d_0\in \mathcal{S}$, while none of $p^{t}d_0,\,\,p^{t+1}d_0,\,\,p^{t+2}d_0$ and $p^{t+3}d_0$ are in $\mathcal{S}$.  Then it must be the case that $t_0\geq t+4>r_p$.  Locally at $p$, we have $L_p\cong\lan a,b^\beta,c^\gamma\ran$ with $a,b,c\in \Z_p^\times$, $0\leq \beta\leq \gamma$ and $\beta+\gamma=t_0>r_p$.  From Corollary \ref{notsplit} we know that $L_p$ is not split by $\mathbb{H}$ and therefore spinor regularity is preserved under the transformation $\lambda_p$, giving 
\[
\lambda_p(L_p)\cong \begin{cases}
\lan a,b,p^{\gamma-2}\ran& \text{ when $\beta=0$ and $\gamma\geq 2$}\\
\lan b,pa,p^{\gamma-1}c\ran & \text{ when $\beta =1$}\\
\lan a,p^{\beta-2}b,p^{\gamma-2}c\ran & \text{ when $\beta\geq 2$},
\end{cases}
\]
which in any case is spinor regular, and 
\[
\ord_p(d(\lambda_p(L_p)))=\begin{cases}
t_0-2& \text{ when $\beta=0$ and $\gamma\geq 2$}\\
t_0-1 & \text{ when $\beta =1$}\\
t_0-4 & \text{ when $\beta\geq 2$}.
\end{cases}
\]
But if $s=\ord_p(d(\lambda_p(L_p)))$, then we have $t\leq s<t_0$, contradicting the minimality of $t_0$.
\end{proof}

\begin{lemma}\label{even}
Suppose $d_0$ is an odd integer.  If there exists an integer $t>2$ such that $2^{t}d_0,\,\,2^{t+1}d_0,\,\,2^{t+2}d_0$ and $2^{t+3}d_0$ are all not in $\mathcal{S}$ or $\mathcal{R}$, then $2^{t_0}d_0\not\in \mathcal{S}$ for all $t_0\geq t$.  
\end{lemma}

\begin{proof}
If the result is not true, then there exists some smallest integer $t_0$ exceeding $t$ for which $2^{t_0}d_0$ lies in $\mathcal S$.  Let $L$ be a spinor regular lattice with the corresponding discriminant, $dL=2^{t_0+1}d_0$.  By the assumption that none of $2^{t}d_0,\,\, 2^{t+1}d_0,\,\,2^{t+2}d_0$ and $2^{t+3}d_0$ lie in $\mathcal{S}$ it must be the case that $t_0\geq t+4\geq 7$; hence $t_0+1\geq t+5\geq 8$.  Therefore it follows from Corollary \ref{notsplit} that $L_2$ is not split by $\mathbb H$ and hence $\lambda_2(L)$ is spinor regular with $d(\lambda_2(L))=2^{t_0+1-s}d_0$ where $s\in \{1,4\}$.  But in either case $t_0-s<t_0$, contradicting the minimality of $t_0$. 
\end{proof}

\section{Enumerating the spinor regular discriminants}\label{results}

The proofs that follow will rely on Tables \ref{App:Tab3by5}, \ref{App:Tab3by7} \ref{App:Tab2by3}, \ref{App:Tab2by3by5}, \ref{App:Tab2by3by7b}, and \ref{App:Tab2by5by7}, given in the Appendix.  These tables provide an accounting of discriminants known to appear in $\mathcal R$ or $\mathcal S$, and indicate discriminants that extend beyond Jagy's search bound in \cite{J}.  Throughout this section, $\btu$ will be used to denote a non-square unit in the local ring $\Z_p$.  

\begin{lemma}\label{Lem:1}
The only prime powers appearing in $\mathcal S$ are $7^3$ and $2^t$ with $t\in \{6,7,8,10,11,12,14\}$
\end{lemma}

\begin{proof}
Since the largest power of 2 appearing in $\mathcal{R}$ is 13, and $2^{15},2^{16},2^{17}$ and $2^{18}$ are all less than 575,000 but don't appear in Table 1, the conclusion follows immediately from Lemma \ref{even} for powers of 2. 

From Lemma \ref{primefactor} we know that if the power of any odd prime, $p$, appears in $\mathcal{S}$ then $p\in \{3,5,7,11,13\}$.  From Lemma \ref{order} we can immediately rule out all powers of 5,7,11 and 13 except $7^3$ which appears in Table \ref{Tab:1} since all admissible powers fall below the bound of 575,000.  The powers of 3 appearing in $\mathcal R$ are $3,3^2,3^3,3^4,3^5$.  But since $3^9<575,000$ and $3^6,3^7,3^8,3^9\not\in \mathcal S\cup \mathcal R$, we conclude from Lemma \ref{odd} that no powers of 3 appear in $\mathcal S$. 
\end{proof}

%******************************************************************************************
%******************************************************************************************

%******************************************************************************************
%
%			THIS PROOF HAS BEEN EDITED AND CORRECTED ON 10 NOV. 2017
%
%******************************************************************************************

\begin{lemma}\label{Lem:3}
For odd primes $p,q$, there are no elements of the form $p^k\cdot q^m$ in $\mathcal{S}$ with $k,m>0$.  
\end{lemma}

\begin{proof}
If $p^k\cdot q^m\in \mathcal S$, then from Lemma \ref{primepair} we know that $(p,q)$ must come from among 
\[
\{(3,5),(3,7),(3,11),(3,13),(5,7),(11,13)\}.
\]
We will consider these prime pairs one at a time. 

When $(p,q)=(3,5)$ then in view of Lemma \ref{order} we know that we only need to check products of the form $3^k\cdot 5^m$ for $m\leq 3$.  When $m=1$ or 2, then $3^k\cdot 5^m<575,000$ and $3^k\cdot 5^m\not\in \SR$ for $k=5,6,7,8$ so we may conclude from Lemma \ref{odd} that $3^k\cdot 5\not\in \mathcal S$ and $3^k\cdot 5^2\not\in \mathcal S$ for any choice of $k$.  When $m=3$, then $3^k\cdot 5^3$ will only exceed 575,000 when $k=8$.  Suppose that $ 3^8\cdot 5^3\in \mathcal {S}$, so there is a spinor regular (but not regular) lattice $L$ with discriminant $2\cdot 3^8\cdot 5^3$.  From Lemma \ref{order} if follows that $L$ has the local structure $L_5\cong \lan a,5b,5^2c\ran$ with $a,b,c\in \Z_5^\times$, and hence is not split by $\mathbb H$.  Consequently $(\lambda_5(L))_5\cong \lan b,5a,c\ran$ is spinor regular with discriminant $2\cdot 3^8\cdot 5$, but we already know that $3^8\cdot 5\not\in \SR$ and $ 3^8\cdot 5<575,000$, a contradiction.  Therefore, $3^8\cdot 5^3\not\in \SR$ and now it follows from Lemma \ref{odd} that $3^k\cdot 5^3\not\in \SR$ for any choice of $k$. 

When $(p,q)=(3,7)$ then from Lemma \ref{order} we know that we only need to check products of the form $3^k\cdot 7^m$ for $m=1,2,3,5$.  When $m=1$ or 2, then $3^k\cdot 7^m\not\in \SR$ and $3^k\cdot 7^m<575,000$ for $k=5,6,7,8$ and so by Lemma \ref{odd} we know that $3^k\cdot 7\not\in \mathcal S$ and $3^k\cdot 7^2\not\in \mathcal S$ for any choice of $k$.  When $m=3$ then we only need to rule out $3^7\cdot 7^3$, and then the remaining $3^k\cdot 7^3\not\in \mathcal S$ follow from Lemma \ref{odd}.  If $L$ is a spinor regular (but not regular) lattice with $dL=2\cdot 3^7\cdot 7^3$ then $L$ is not split by $\mathbb H$ at 3 by Corollary \ref{notsplit} and consequently $\lambda_3(L)$ is spinor regular with $d(\lambda_3(L))=2\cdot 3^{7-t}\cdot 7^3$ with $t\in \{1,2,4\}$, but we already know that this is not possible since $3^{7-t}\cdot 7^3<575,000$ and $3^{7-t}\cdot 7^3\not\in \SR$ for $t\in \{1,2,4\}$.  Therefore, $3^{7}\cdot 7^3\not\in \mathcal S$ and it follows from Lemma \ref{odd} that $3^k\cdot 7^3\not\in \mathcal S$ for any $k$.  When $m=5$ we only need to rule out $3^k\cdot 7^5$ for $k=4,5,6,7$ to rule out all remaining cases.  But for each of these, we know that $L_7$ is not split by $\mathbb H$, and consequently $\lambda_7(L)$ is spinor regular with $d(\lambda_7(L))=2\cdot 3^k\cdot 7^{5-t}$ for $t\in \{2,4\}$.  But all such discriminants are less than 575,000 and do not appear in $\SR$, and so no such $L$ can exist.  Therefore $3^k\cdot 7^m\not\in \mathcal S$. 

When $(p,q)=(3,11),(3,13),(5,7)$ and $(11,13)$ the argument follows just as above and all discriminants beyond the bound 575,000 can be ruled out by applying Corollary \ref{notsplit} followed by Lemma \ref{descent} and showing that the resulting discriminant is impossible, and then ruling out all remaining cases with Lemma \ref{odd}. 
\end{proof}

%******************************************************************************************
%******************************************************************************************

%******************************************************************************************
%
%			THIS PROOF HAS BEEN EDITED AND CORRECTED ON 10 NOV. 2017
%
%******************************************************************************************

\begin{lemma}
For an odd prime $p$, the only $2^k\cdot p^m$ with $k,m>0$ occurring in $\mathcal S$ are those with $p=3,7$ appearing in Table \ref{Tab:1}. 
\end{lemma}

\begin{proof}
By Lemma \ref{odd} we can immediately rule out all $2^k\cdot 3^m$ not appearing in Table \ref{Tab:1} when $k\leq 5$.  We only need to show that $2^k\cdot 3^m\not\in \mathcal S$ for $(k,m)=(6,9),(12,5),(13,4),(14,4),(15,3),(16,3),(16,2),(17,2)$ and then we can use Lemmas \ref{odd} and \ref{even} to rule out all remaining cases. 

To eliminate $2^6\cdot 3^m$ we only need to rule out $2^6\cdot 3^9$ and then all remaining possibilities can be eliminated using Lemma \ref{odd}.  Suppose that $L$ is a spinor regular (but not regular) lattice with $dL=2^7\cdot 3^9$.  Then from Lemma \ref{order} we know that $L_3\cong \lan a,3^\beta b,3^\gamma c\ran $ where $2\leq \beta\leq \gamma$ and $a,b,c\in \Z_3^\times$. Consequently $\lambda_3(L)$ is spinor regular and $d(\lambda_3(L))=2^7\cdot 3^5$ and hence corresponds to the form $[8,9,56,0,8,0]$ from \cite{JKS}.  This means that locally at 2 we have 
\[
(\lambda_3(L))_2\cong 2\lan 1\ran \perp 2^3\mathbb A,
\]
and hence $L_2$ is not split by $\mathbb H$. But then $\lambda_2(L)$ must be spinor regular and $d(\lambda_2(L))=2^{7-t}\cdot 3^9$ with $t\in \{1,4\}$, which we know does not occur.  Therefore, $2^6\cdot 3^9\not\in \mathcal S$ and using Lemma \ref{odd} we can rule out all $2^6\cdot 3^m$.  At this point we can easily rule out $2^7\cdot 3^m$ using Lemma \ref{odd}, and then $2^k\cdot 3^m$ for $m\geq 6$.

Next we must show that $2^{12}\cdot 3^5\not\in \mathcal S$.  Suppose that $L$ is spinor regular (but not regular) with $dL=2^{13}\cdot 3^5$.  Then from \cite[Proposition 5.5]{CE} we know that $L_2$ is not split by $\mathbb H$, and consequently $\lambda_2(L)$ is spinor regular with $d(\lambda_2(L))=2^{9}\cdot 3^5$. Consequently $\lambda_2(L)$ corresponds to the form $[9,17,113,-14,6,6]$ from \cite{JKS}, and hence
\[
(\lambda_2(L))_3\cong \lan 1\ran\perp 3^2\lan 1 \ran \perp 3^3\lan \btu\ran,
\]
implying 
\[
L_3\cong \lan a\ran\perp 3^2\lan a \ran \perp 3^3\lan b\ran,
\]
where one of $a,b$ is a square in $\Z_3^\times$. Therefore, $L_3$ is not split by $\mathbb H$ and hence $\lambda_3(L)$ is spinor regular with $d(\lambda_3(L))=2^{13}\cdot 3$ and must correspond to one of $[7,20,23,-4,2,4]$ or $[7,23,23,-18,2,2]$ from \cite{JKS}, both of which have the local structure 
\[
(\lambda_3(L))_3\cong \lan \btu,\btu\ran\perp 3\lan \btu \ran,
\]
implying 
\[
L_3\cong \lan \btu\ran \perp 3^2\lan \btu\ran\perp 3^3\lan \btu \ran,
\]
a contradiction. Therefore, $2^{12}\cdot 3^5\not\in \mathcal S$, and from here we can use Lemma \ref{even} to rule out all $2^k\cdot 3^5$. 

The product $2^{13}\cdot 3^4$ can be easily ruled out by Lemma \ref{descent} since $2^{t}\cdot 3^4\not\in \SR$ for $t=9,11$ or $12$.  Now it only remains to eliminate $2^{14}\cdot 3^4$, $2^{15}\cdot 3^3$, $2^{16}\cdot 3^3$, $2^{16}\cdot 3^2$ and $2^{17}\cdot 3^2$.  In each of these cases, for the corresponding lattice $L$, we know from Corollary \ref{notsplit} that $L_2$ is not split by $\mathbb H$ and consequently $\lambda_2(L)$ is spinor regular and must correspond to one of the forms [16,19,76,4,16,8], [13,16,52,16,4,8], [13,21,37,6,10,6], [17,20,57,12,6,4], [17,32,41,16,10,16], [23,32,47,16,22,16], [11,32,44,-16,4,8]$,  $[11,35,44,28,4,10], [15,20,56,16,0,12] or [15,23,44,-4,12,6] from \cite{JKS}. Using SageMath \cite{sage} we easily see that none of these forms are split by $\mathbb H$ at 3, and consequently in each case $\lambda_3(L)$ is also spinor regular. But $\ord_2(d(\lambda_3(L)))\geq 14$, which we know cannot happen, thus ruling out all remaining cases. 

Proofs for $p=5,7,11,13$ proceed similarly. 
\end{proof}

%******************************************************************************************
%******************************************************************************************

\begin{lemma}\label{lemma235}
There are no elements of the form $2^k\cdot 3^m\cdot 5^n$ in $\mathcal S$ with $k,m,n>0$. 
\end{lemma}

\begin{proof}
From Lemma \ref{order} we know that $2^k\cdot 3^m\cdot 5^n$ can only be in $\mathcal S$ when $n\leq 3$.  To begin, we will consider discriminants of the form $2^k\cdot 3^m\cdot 5$.  From Lemma \ref{odd} we can immediately eliminate all $2^k\cdot 3^m\cdot 5$ for $k=1,2,3,4,5,7$.  Now we only need to show that $2^k\cdot 3^m\cdot 5\not\in \mathcal S$ for $(k,m)=(6,7),(6,8),(6,9),(9,5),(10,5),(11,4),(12,4),(13,3),$ $(14,3),(14,2)$ and then we can use Lemmas \ref{odd} and \ref{even} to eliminate all remaining $2^k\cdot 3^m\cdot 5$. 

Suppose that $L$ is a spinor regular (but not regular) lattice with $dL=2^{k+1}\cdot 3^m\cdot 5$.  We can immediately rule out $(k,m)=(9,5),(11,4),(13,3)$, since in these cases $L_2$ is not split by $\mathbb H$ and hence $\lambda_2(L)$ is spinor regular with $d(\lambda_2(L))=2^{k+1-t}\cdot 3^m\cdot 5$ for $t\in \{1,4\}$, which we know cannot happen. 

In cases where $(k,m)=(6,m)$, then $dL=2^{7}\cdot 3^m\cdot 5$ and from Lemma \ref{order} it follows that $L_3\cong \lan a,3^\beta b,3^\gamma c\ran$ where $2\leq \beta\leq \gamma$ and $a,b,c\in \Z_3^\times$.  Hence $L_3$ is not split by $\mathbb H$, and $d(\lambda_3(L))=2^{k+1}\cdot 3^{m-4}\cdot 5$.  Therefore, $\lambda_3(L)$ corresponds to one of [12,13,21,6,12,12], [13,13,13,2,2,2], [5,5,92,-4,4,2], [7,18,19,6,2,6] or [7,19,19,14,2,2] when $m=7$, [8,23,39,6,0,8] when $m=8$ and [13,28,61,28,2,4] when $m=9$, none of which are split by $\mathbb H$ at 2. Therefore in any case $\lambda_2(L)$ is spinor regular with $d(\lambda_2(L))=2^{3}\cdot 3^m\cdot 5$, which we know cannot happen. 

In all of the remaining cases we know that $k\geq 8$ and so $L_2$ is not split by $\mathbb H$.  Consequently in all of these cases $\lambda_2(L)$ is spinor regular.  When $(k,m)=(10,5)$ then $dL=2^{11}\cdot 3^5\cdot 5$, and we know from Lemma \ref{order} that $L_3$ is not split by $\mathbb H$ and has a splitting
\[
L_3\cong\begin{cases}
\lan a,3b,3^4c\ran\\
\lan a,3^2b, 3^3c\ran
\end{cases}
\]
for $a,b,c\in \Z_3^\times$, and hence $\lambda_3(L)$ is spinor regular with
\[
(\lambda_3(L))_3\cong\begin{cases}
\lan b,3a,3^3c\ran\\
\lan a,b, 3c\ran,
\end{cases}
\]  
but we can easily eliminate the first case since $2^{11}\cdot 3^4\cdot 5\not\in \SR$.  Therefore $d(\lambda_3(L))=2^{11}\cdot 3\cdot 5$ and hence corresponds to one of $[11,19,19,6,2,2]$ or $[7,23,31,-22,6,2]$ from \cite{JKS}.  Both of these forms admit local structure $\lan 1,1,3\ran $ at 3, and so from here we know that $L_3\cong \lan 1,3^2,3^3\ran$.  On the other hand, we know that $\lambda_2(L)$ is spinor regular with discriminant $2^{7}\cdot 3^5\cdot 5$ and hence corresponds to $[13,28,61,28,2,4]$ from \cite{JKS}, which locally at 3 has the structure $\lan \btu, 3^2\btu,3^3\ran$, contradicting what we already know of $L_3$. 

For the remaining cases, we know that $L_3$ must be split by $\mathbb H$ since $\lambda_3(L)$ cannot be spinor regular, and $\lambda_2(L)$ is spinor regular with 
\[
d(\lambda_2(L))=\begin{cases}
2^{9}\cdot 3^4\cdot 5 & \text{ for }(k,m)=(12,4)\\
2^{11}\cdot 3^2\cdot 5 & \text{ for }(k,m)=(14,2)\\
2^{11}\cdot 3^3\cdot 5 & \text{ for }(k,m)=(14,3).
\end{cases}
\]
Hence, the lattice $\lambda_2(L)$ corresponds to one of the forms [23,32,44,32,4,8], [17,32,32,32,8,16], [21,29,29,26,18,18] or [19,28,67,-4,10,4] from \cite{JKS}, none of which are split by $\mathbb H$ at 3. 

Next we will deal with cases of the form $2^k\cdot 3^m\cdot 5^2$.  For any $L$ with such a discriminant there are only two possibilities for the local structure of $L_5$, namely $\lan a,5b,5c\ran$ or $\lan a,b,5^2c\ran$ where $a,b,c\in \Z_5^\times$; we will deal with these cases one at a time.  

First, suppose that $L_5\cong \lan a,5b,5c\ran$, so $L_5$ is not split by $\mathbb H$, implying $\lambda_5(L)$ is spinor regular with $d(\lambda_5(L))=2^{k+1}\cdot 3^m\cdot 5$.  Then in view of the preceding paragraph, the only possibility is that $d(\lambda_5(L))=2^{11}\cdot 3^3\cdot 5$, since otherwise the discriminant is already within our search bounds or $\lambda_5(L)$ is not spinor regular.  Therefore $\lambda_5(L)$ must correspond to $[19,28,67,-4,10,4]$ which has local structures 
\[
(\lambda_5(L))_p\cong \begin{cases}
2\lan 3\ran \perp 2^5\lan3,7\ran & \text{ for }p=2\\
 \lan \btu\ran \perp 3\lan 1\ran \perp 3^2\lan \btu\ran& \text{ for }p=3,
\end{cases}
\]
and consequently the binary component of $L_2$ has a non-square discriminant.  It also follows that $L_3$ is not split by $\mathbb H$ and hence $\lambda_3(L)$ is spinor regular with $d(\lambda_3(L))=2^{11}\cdot 3^2\cdot 5^2$ and hence corresponds to one of $[13,37,132,-36,12,2]$ or $[33,48,52,48,12,24]$ from \cite{JKS}.  The latter of these has local structure $2\lan 1\ran\perp 2^5\lan 1,1\ran$ at 2, which cannot happen since we know the binary part of $L_2$ has a non-square discriminant.  So we may assume $\lambda_3(L)$ corresponds to the former, which has local structure $ \lan \btu\ran \perp 3\lan \btu,\btu\ran$ at 3, implying $L_3\cong  \lan \btu\ran \perp 3\lan \btu\ran \perp 3^2\lan \btu\ran$, which contradicts what we know of $L_3$ from above. Therefore, this case of $L_5\cong \lan a,5b,5c\ran$ cannot occur.

Supposing $L_5\cong\lan a,b,5^2c\ran$, it is possible that $L_5$ is either split by $\mathbb H$ or not.  Consider the case where $L_5$ is not split by $\mathbb H$ and hence $\lambda_5(L)$ is spinor regular with discriminant $2^k\cdot 3^m$ and hence $(k,m)$ must be one of $(8,5),(10,4),(10,3),(11,3),$
$(12,3),(12,2),(13,1)$ or $(13,2)$.  In all of these cases $L$ is not split by $\mathbb H$ at 2, and consequently $\lambda_2(L)$ is spinor regular with $d(\lambda_2(L))=2^{k+1-t}\cdot 3^m\cdot 5^2$ with $t\in \{1,4\}$.  With this, we can immediately rule out $(k,m)=(8,5),(13,1)$ since the corresponding discriminants are not in $\SR$. When $(k,m)=(10,3),(10,4),(12,2)$ then $\lambda_2(L)$ must correspond to one of [11,26,39,6,6,2], [11,35,39,-30,6,10], [25,25,28,-20,20,10], [8,17,92,4,8,8] or [9,41,41,-38,6,6] for $(k,m)=(10,3)$; or, [27,40,43,40,18,0] for $(k,m)=(10,4)$; or, [3,40,120,0,0,0], [12,33,40,0,0,12], [13,33,37,-18,2,6],$ $ [16,19,64,8,16,16] or [27,28,28,-24,12,12] for $(k,m)=(12,2)$.  However, in all of these cases $(\lambda_2(L))_5\cong \lan a,5b,5c \ran$ for $a,b,c\in \Z_5^\times$, a contradiction.  Having ruled out these cases, we can now easily eliminate $(k,m)=(11,3),(13,2)$.  In the last remaining case $(k,m)=(12,3)$, we have $d(\lambda_2(L))=2^{9}\cdot 3^3\cdot 5^2$ which corresponds to one of $[9,41,120,0,0,6],$ $[11,35,120,0,0,10],[25,48,52,48,20,0]$ or $[36,39,44,12,24,36]$ from \cite{JKS}, but again, in all of these cases $(\lambda_2(L))_5\cong \lan a,5b,5c \ran$ for $a,b,c\in \Z_5^\times$.

Hence we have ruled out all cases where $L_5$ is of the form $\lan a,b,5^2c\ran$ and not split by $\mathbb H$, and so we will turn to the case where $L_5$ is split by $\mathbb H$.  Since $L$ is spinor regular but not regular we may assume that either $L_2$ or $L_3$ is not split by $\mathbb H$, and consequently one of $\lambda_2(L)$ or $\lambda_3(L)$ is spinor regular and must in fact correspond to one of the forms in \cite{JKS}.  However, by an exhaustive computation using SageMath \cite{sage} we show that every form in \cite{JKS} with discriminant of interest has the structure $\lan a\ran \perp 5\lan b,c\ran$ locally at 5, where $a,b,c\in \Z_5^\times$, a contradiction.  Therefore we can easily rule out all of these cases. 

Next we consider $dL=2^{k+1}\cdot 3^m\cdot 5^3$.  In these cases we know from Lemma \ref{order} that $L_5\cong \lan a,5b,5^2c \ran $ with $a,b,c\in \Z_5^\times$.   Therefore, $L_5$ is not split by $\mathbb H$ and hence $\lambda_5(L)$ is spinor regular with $d(\lambda_5(L))=2^{k+1}\cdot 3^m\cdot 5^2$.  Consequently we only need to rule out cases where $(k,m)=(6,4),(6,5),(8,3),(8,4),(10,1),(10,2)$.  When $(k,m)=(6,4),(6,5)$, then $\lambda_5(L)$ corresponds to $[27,40,43,40,18,0]$,$[9,41,281,-38,6,6]$, respectively, neither of which are split by $\mathbb H$ at 2.  Therefore, $\lambda_2(L)$ is spinor regular, which cannot occur.  At this point we can immediately rule out $(k,m)=(8,3),(8,4)$ since these are both not split by $\mathbb H$ at 2, but again it is not possible that $\lambda_2(L)$ is spinor regular in these cases.  

For $(k,m)=(10,1)$ we know that $\lambda_5(L)$ corresponds to one of $[11,16,124,16,4,8]$ or $[16,31,44,4,16,8]$ which have local structures 
\[
(\lambda_5(L))_5\cong \begin{cases}
\lan \btu\ran\perp 5\lan 1,\btu\ran   \\
\lan \btu\ran\perp 5\lan1, \btu\ran 
\end{cases}
\]
and 
\[
(\lambda_5(L))_3\cong \begin{cases}
\lan 1,1\ran\perp 3\lan \btu\ran  \\
\lan \btu,\btu\ran\perp 3\lan \btu\ran,   
\end{cases}
\]
respectively.   On the other hand, we also know that $L_2$ is not split by $\mathbb H$, and so $\lambda_2(L)$ must correspond to one of $[21,21,21,2,18,18],[5,8,152,8,0,0]$ or $[5,24,56,24,0,0]$ which have local structures 
\[
(\lambda_2(L))_5\cong \begin{cases}
\lan \btu\ran\perp 5\lan 1\ran\perp 5^2\lan \btu\ran   \\
\lan 1\ran\perp 5\lan \btu\ran\perp 5^2\lan \btu\ran   \\
\lan \btu\ran\perp 5\lan \btu\ran\perp 5^2\lan 1\ran,   
\end{cases}
\]
and
\[
(\lambda_2(L))_3\cong \begin{cases}
\lan 1,\btu\ran\perp 3\lan \btu\ran  \\
\lan 1,1\ran\perp 3\lan 1\ran   \\
\lan 1,1\ran\perp 3\lan 1\ran,   
\end{cases}
\]
respectively.  We can rule out the first case, since the leading binary component of $L_3$ must have a square discriminant. We can also rule out the second case, given what we know about the structure of $(\lambda_5(L))_5$.  Therefore we can conclude that $\lambda_2(L)$ corresponds to $[5,24,56,24,0,0]$ and consequently 
\[
L_p\cong \begin{cases}
2\lan 5\ran\perp 2^5\mathbb A & \text{ when }p=2\\
 \lan a,a\ran\perp 3\lan a\ran & \text{ when }p=3\\
\lan \btu\cdot b\ran\perp 5\lan \btu\cdot b\ran\perp 5^2\lan b\ran & \text{ when }p=5\\
\lan 1,1,2^{11}\cdot 3\cdot 5^3\ran & \text{ otherwise}. 
\end{cases}
\] 
where $a\in \Z_3^\times$ and $b\in \Z_5^\times$. But now $\Z_p^\times\subseteq \theta(O^+(L_p))$ at every prime $p$ [cf. \cite[Lemma 1]{H75}] and so it follows from \cite[102:9]{OM} that $L$ was regular to begin with, a contradiction. 

Finally, when $(k,m)=(10,2)$ then the lattice $\lambda_2(L)$ corresponds to one of the forms [15,24,56,24,0,0], [7,52,52,-16,4,4] or [8,15,152,0,8,0] all of which have the $3$-adic local structure $\lan a,3b,3c\ran$ for $a,b,c\in \Z_3^\times$.  Consequently $\lambda_3(L)$ is spinor regular with $d(\lambda_3(L))=2^{11}\cdot 3\cdot 5^3$, which we know cannot occur.  
\end{proof}

%******************************************************************************************
%******************************************************************************************

\begin{lemma}\label{lemma237}
There are no elements of the form $2^k\cdot 3^m\cdot 7^n$ in $\mathcal S$ with $k,m,n>0$. 
\end{lemma}

\begin{proof}
From Lemma \ref{order} we know that $2^k\cdot 3^m\cdot 7^n$ can only appear in $\mathcal S$ when $n\leq 5$.  Suppose that $L$ is a spinor regular (but not regular) lattice with $dL=2^{k+1}\cdot 3^m\cdot 7^n$.  When $n=1$ then we only need to rule out discriminants with $(k,m)$ coming from among $(6,7),(6,8),(10,4),(12,3),(13,3),(14,2),(14,3)$, and then all remaining cases can be ruled out by Lemmas \ref{odd} and \ref{even}.  

If $(k,m)=(6,7),(6,8)$ then we may assume that $L_2$ is split by $\mathbb H$, or else $\lambda_2(L)$ is spinor regular with $d(\lambda_2(L))=2^{3}\cdot 3^m\cdot 7$ which we already know can't happen.  Moreover, we know from Corollary \ref{notsplit} that $L_3$ is not split by $\mathbb H$, and therefore $\lambda_3(L)$ is spinor regular with $d(\lambda_3(L))=2^{7}\cdot 3^m\cdot 7$ with $m=3,4$ and hence corresponds to one of $[4,31,31,26,4,4],[8,20,23,4,8,8]$, or $[13,28,28,-16,4,4]$ from \cite{JKS}.  But none of these are split by $\mathbb H$ at 2, and therefore this cannot happen.

Suppose that $(k,m)=(10,4)$, in which case we know that $L_2$ is not split by $\mathbb H$, and hence $\lambda_2(L)$ is spinor regular with $d(\lambda_2(L))=2^{7}\cdot 3^4\cdot 7$ and corresponds to $[13,28,28,-16,4,4]$ from \cite{JKS}.  Therefore, we have 
\[
(\lambda_2(L))_p\cong \begin{cases}
2\lan 5\ran \perp 2^3\mathbb A & \text{ for }p=2\\
\lan \btu\ran \perp 3^2\lan \btu,\btu\ran  & \text{ for }p=3\\
\lan 1,\btu\ran \perp 7\lan \btu\ran  & \text{ for }p=7\\
\lan 1,1,2^7\cdot 3^4\cdot 7\ran & \text{ otherwise},
\end{cases}
\]
and consequently $\Z_p^\times\subseteq \theta(O^+(L_p))$ for every $p$ and hence $L$ was regular to begin with, a contradiction.

If $(k,m)=(14,2)$ then we may assume that $L_3$ is split by $\mathbb H$, or else $\lambda_3(L)$ would be spinor regular with $d(\lambda_3(L))=2^{15}\cdot 3^{2-t}\cdot 7$ with $t\in \{1,2\}$, neither of which can occur in $\SR$. Moreover, we know that $L_2$ is not split by $\mathbb H$, and consequently $\lambda_2(L)$ is spinor regular with $d(\lambda_2(L))=2^{11}\cdot 3^2\cdot 7$.  Therefore, $\lambda_3(L)$ corresponds to $[16,16,23,8,8,0]$ from \cite{JKS} which has local structure 
\[
(\lambda_2(L))_3\cong \lan\btu,\btu \ran \perp 3\lan \btu\ran,
\]
contradicting the fact that $L_3$ is split by $\mathbb H$. 

For the remaining cases $(k,m)=(12,3), (13,3),(14,3)$ we know that $L_2$ is not split by $\mathbb H$, and we may assume that $L_3$ is split by $\mathbb H$ or else $\lambda_3(L)$ would be spinor regular with $d(\lambda_3(L))=2^{k}\cdot 3^{3-t}\cdot 7$ where $t\in \{1,2\}$, which we know cannot happen.  When $(k,m)=(12,3)$ then $\lambda_2(L)$ is spinor regular with $d(\lambda_2(L))=2^{9}\cdot 3^3\cdot 7$ and hence corresponds to one of $[15,16,55,16,6,0]$ or $[15,23,44,-20,12,6]$ from \cite{JKS}, however both of these have the 3-adic local structure $\lan a,3b,3^2 c\ran$ with $a,b,c\in \Z_3^\times$, a contradiction.  When $(m,k)=(13,3)$ then $\lambda_2(L)$ cannot be spinor regular so this case cannot occur.  When $(m,k)=(14,3)$, then $\lambda_2(L)$ is spinor regular with $d(\lambda_2(L))=2^{11}\cdot 3^3\cdot 7$ and hence corresponds to one of $[16,31,103,-10,8,8]$ from \cite{JKS}, which has the 3-adic local structure $\lan a,3b,3^2 c\ran$ with $a,b,c\in \Z_3^\times$, a contradiction.

For $2^k\cdot 3^m\cdot 7^2$, we know from Corollary \ref{notsplit} that any associated spinor regular form must not be split by $\mathbb H$ at 7, and hence $\lambda_7(L)$ is spinor regular with discriminant $2^k\cdot 3^m\cdot 7$.    Therefore the only possibility is that $\lambda_7(L)$ has discriminant $2^{10}\cdot 3^3\cdot 7$ and hence corresponds to the form $[16,31,103,-10,8,8]$.  However, locally at 7 this has the structure 
\[
(\lambda_7(L))_7\cong\lan 1,1\ran\perp 7\lan \btu\ran
\]
and consequently 
\[
L_7\cong \lan 1\ran \perp 7\lan 1,\btu\ran.
\]
On the other hand, we also know that $L_2$ is not split by $\mathbb H$ and $\lambda_2(L)$ is spinor regular with $d(\lambda_2(L))=2^7\cdot 3^3\cdot 7^2$ and hence corresponds to one of $[12,17,129,6,12,12]$ or $[12,28,73,28,12,0]$.  However, locally at 7 these both have the structure 
\[
(\lambda_2(L))_7\cong \lan \btu\ran\perp 7\lan 1,1\ran
\]
contradicting the fact that we know the binary part of $L_7$ has a non-square discriminant.  Therefore this case cannot occur.

For $2^k\cdot 3^m\cdot 7^3$, we know from Lemma \ref{order} that $L_7\cong \lan a,7b, 7^2c \ran$ where $a,b,c\in \Z_7^\times$ is not split by $\mathbb H$.  Consequently $\lambda_7(L)$ is spinor regular with discriminant $2^k\cdot 3^m\cdot 7^2$ and hence in this case we only need to rule out $(k,m)$ among $(6,3),(6,4),(8,2),(8,3),$ $(10,1),(10,2)$ or $(10,3)$.  If $(k,m)=(6,3),(6,4)$ then we may conclude that $L_2$ is split by $\mathbb H$, since otherwise $\lambda_2(L)$ would be spinor regular, which cannot happen.  But we know that $L_7$ is not split by $\mathbb H$, and in these cases $\lambda_7(L)$ has discriminant $2^7\cdot 3^m\cdot 7^2$ with $m=3,4$ and correspond to the forms $[12,17,129,6,12,12],[12,28,73,28,12,0]$ or $[19,27,136,-24,16,6]$, all of which contain an anisotropic plane.  Therefore $L_2$ cannot be split by $\mathbb H$ and hence these cases cannot occur.   

Now we can immediately rule out $(k,m)=(8,2),(8,3),(10,3)$ since in these cases $L_2$ is not split by $\mathbb H$ and consequently $\lambda_2(L)$ is spinor regular, but such spinor regular discriminants do not occur.  

If $(k,m)=(10,1),(10,2)$ then $L_2$ is not split by $\mathbb H$ and consequently $\lambda_2(L)$ is spinor regular $d(\lambda_2(L))=2^{7}\cdot 3^m\cdot 7^3$ for $m=1,2$.  Hence $\lambda_2(L)$ corresponds to one of $[20,31,31,-22,4,4]$ or $[28,37,60,12,0,28]$ from \cite{JKS}, and we have 
\[
(\lambda_2(L))_7\cong 
\begin{cases}
\lan\btu \ran\perp 7\lan \btu\ran \perp 7^2\lan \btu\ran & \text{ when }m=1 \\
\lan1 \ran\perp 7\lan 1\ran \perp 7^2\lan 1\ran & \text{ when }m=2.
\end{cases}
\]
Therefore, $L_7$ is not split by $\mathbb H$ and consequently $\lambda_7(L)$ is spinor regular with discriminant $2^{11}\cdot 3^m\cdot 7^2$ thus corresponding to one of $[17,41,68,-36,12,10]$ or $[16,43,172,4,16,8]$ which have the corresponding local structures
\[
(\lambda_7(L))_7\cong 
\begin{cases}
\lan\btu \ran\perp 7\lan 1,1\ran & \text{ when }m=1 \\
\lan1 \ran\perp 7\lan \btu,\btu\ran & \text{ when }m=2,
\end{cases}
\]
and hence 
\[
L_7\cong \begin{cases}
\lan\btu \ran\perp 7\lan 1\ran \perp7^2\lan1\ran & \text{ when }m=1 \\
\lan1 \ran\perp 7\lan \btu\ran \perp 7^2\lan \btu\ran & \text{ when }m=2,
\end{cases}
\]
a contradiction, given what we already determined about the units appearing in the splitting of $L_7$. 

Finally, we must deal with the case of $2^k\cdot 3^m\cdot 7^5$.  Any lattice $L$ associated to a spinor regular form of discriminant $2^k\cdot 3^m\cdot 7^5$ will necessarily be split by $\mathbb H$ at 7 by Corollary \ref{notsplit}.  Moreover, we know from Lemma \ref{order} that such a lattice will have a local structure $L_7\cong \lan a,7^2b,7^3c\ran$ and consequently $\lambda_7(L)$ is spinor regular and $(\lambda_7(L))_7\cong\lan a,b,7c\ran $.  Consequently we can immediately rule out any $(k,m)$ where $2^k\cdot 3^m\cdot 7\not\in \SR$.  The strategy now will be to rule out all $2^k\cdot 3^m\cdot 7^5$ where $1\leq k\leq 6$ and $1\leq m\leq 3$.  Beyond these values, we know that neither $L_2$ nor $L_3$ is split by $\mathbb H$, and consequently we can systematically descend to one of these values.  We can immediately rule out $(k,m)=(1,1),(1,2),(2,1),(3,1)$ since these are below our search bound and not in $\mathcal S$.  We can also rule out $k=5$ and $m=5$, since in these cases $2^k\cdot 3^m\cdot 7\not\in \SR$.  Using SageMath \cite{sage} and the forms presented in \cite{JKS} we find none of the $(\lambda_7(L))_2$ are split by $\mathbb H$, and therefore we may conclude that $L_2$ is never split by $\mathbb H$.  This allows us to eliminate $(k,m)=(4,1),(6,1)$ since under $\lambda_2$ these must descend to forms which are spinor regular, but this is impossible.  Next we use SageMath \cite{sage} to examine the local structure of $\lambda_7(L)$ for each of the remaining lattices, and we find that in all cases $(\lambda_7(L))_3$ is not split by $\mathbb H$.  Whenever $L_3$ is not split by $\mathbb H$, we know that $\lambda_3(L)$ should descend to a spinor regular form with discriminant $2^k\cdot 3^{m-t}\cdot 7^5$ where $t=1,2,4$, but we know this cannot happen.  
\end{proof}

%******************************************************************************************
%******************************************************************************************

\begin{lemma}\label{lemma257}
There are no elements of the form $2^k\cdot 5^m\cdot 7^n$ in $\mathcal S$ with $k,m,n>0$. 
\end{lemma}

\begin{proof}
From Lemma \ref{order} we know that $2^k\cdot 5^m\cdot 7^n$ can only appear in $\mathcal S$ provided that $m\leq 3$ and $n\leq 5$.  Suppose that $L$ is a spinor regular (but not regular) form with $dL=2^{k+1}\cdot 5^m\cdot 7^n$.  All discriminants of the form $2^k\cdot 5^m\cdot 7$ can be immediately ruled out by Lemmas \ref{odd} and \ref{even}.  

When $dL=2^{k+1}\cdot 5^m\cdot 7^2$ we only need to rule out cases where $(k,m)=(9,2),(10,2)$ and then all remaining cases follow by applying Lemmas \ref{odd} and \ref{even}.  In these cases we know from Lemma \ref{order} that $L_7\cong \lan a,7b,7c\ran$ with $a,b,c\in \Z_7^\times$ and consequently $\lambda_7(L)$ is spinor regular and $(\lambda_7(L))_7\cong \lan b,7a,c\ran $.  But we already know that such a lattice cannot be spinor regular when $(k,m)=(9,2),(10,2)$.

If $dL=2^{k+1}\cdot 5^m\cdot 7^3$ we know that $L_7\cong \lan a,7b,7^2c\ran $ with $a,b,c\in \Z_7^\times$, and hence $\lambda_7(L)$ is spinor regular and $(\lambda_7(L))_7\cong \lan b,7a,7c\ran $ is spinor regular.  However, we have already ruled out all relevant cases in the preceding paragraph. 

When $dL=2^{k+1}\cdot 5^m\cdot 7^5$, then $L_7\cong \lan a,7^2b,7^3c\ran $ for $a,b,c\in \Z_7^\times$ and consequently $\lambda_7(L)$ is spinor regular and $(\lambda_7(L))_7\cong \lan a,b,7c\ran$.  Therefore in this case we only need to eliminate $(k,m)=(1,2),(2,2),(6,1),(6,2)$.  For $(k,m)=(1,2),(2,2)$, $\lambda_7(L)$ must correspond to one of $[3,3,10,0,0,1]$ and $[5,6,6,2,0,0]$ from \cite{JKS}.  But for each of these forms the associated lattice has the local structure 
\[
(\lambda_7(L))_5\cong \lan a\ran \perp 5\lan b,c\ran 
\]
for $a,b\in \Z_5^\times$, and consequently $L_5$ is not split by $\mathbb H$ so $\lambda_5(L)$ is spinor regular with
\[
d(\lambda_5(L))=\begin{cases}
2\cdot3\cdot 5 & \text{ when }m=1\\
2\cdot3^2\cdot 5 & \text{ when }m=2
\end{cases} 
\]
neither of which can occur. 

If $(k,m)=(6,1),(6,2)$, then $\lambda_7(L)$ corresponds to the forms $[1,8,72,8,0,0]$ and $[5,24,24,8,0,0]$, respectively.  But for each of these forms the associated lattice has the local structure 
\[
(\lambda_7(L))_2\cong 2\lan a\ran \perp 2^3\mathbb A,
\]
for $a\in \Z_2^\times$.  But this means that $L_2$ is not split by $\mathbb H$ and hence $\lambda_2(L)$ should be spinor regular with discriminant $2^2\cdot 5^m\cdot 7^5$ for $m=1,2$, which we know does not happen. 
\end{proof}

Proofs showing that $2^k\cdot 3^m\cdot 11^n$, $2^k\cdot 3^m\cdot 13^n$, and $2^k\cdot 11^m\cdot 13^n$ are never in $\mathcal S$ follow the same arguments as the proofs of Lemmas \ref{lemma235}, \ref{lemma237} and \ref{lemma257}.  In these cases it is actually slightly easier to determine that no such elements exist in $\mathcal S$, because the relative scarcity of elements in $\mathcal R$ of the form $2^k\cdot 3^m\cdot 11^n$, $2^k\cdot 3^m\cdot 13^n$, and $2^k\cdot 11^m\cdot 13^n$ leaves little to be eliminated.  Combining all of the lemmas of this section we have shown that the discriminants appearing in Table 1 are precisely the spinor regular discriminants and therefore Table 1 is complete, thus proving Theorem \ref{main}. 

%******************************************************************************************
%******************************************************************************************

\section{Appendix}
In the tables below``$r$" denotes a product that corresponds to some $disc(f)\in \mathcal{R}$,  ``$s$" denotes a product that corresponds to some $disc(f)\in \mathcal{S}$, ``$*$" denotes a product that is larger than 575,000, ``x" denotes a product that is already ruled out by Lemma \ref{order}, and ``$-$" denotes a product that is smaller than 575,000 but appears in neither $\mathcal{S}$ nor $\mathcal{R}$.

\begin{table}[h]
\begin{tabular}{c|ccc}
\hline
$3^k\cdot 5^m$&5&$5^2$&$5^3$\\
\hline
$3$      &r&r&r\\
$3^2$  &r&r&r\\
$3^3$  &r&r&r\\
$3^4$  &r&r&r\\
$3^5$  &-&-&-\\
$3^6$  &-&-&-\\
$3^7$  &-&-&-\\
$3^8$  &-&-&*\\
$3^9$  &-&-&*\\
\hline
\end{tabular}
\caption{}\label{App:Tab3by5}
\end{table}

\begin{table}[h]
\begin{tabular}{c|ccccc}
\hline
$3^k\cdot 7^m$&7&$7^2$&$7^3$&$7^4$&$7^5$\\
\hline
$3$  &r&r&-&x&-\\
$3^2$  &r&r&-&x&-\\
$3^3$ &r&r&-&x&-\\
$3^4$  &-&-&-&x&*\\
$3^5$  &-&-&-&x&*\\
$3^6$  &-&-&-&x&*\\
$3^7$  &-&-&*&x&*\\
\hline
\end{tabular}
\caption{}\label{App:Tab3by7}
\end{table}

\begin{table}[h]
\begin{tabular}{c|ccccccccc}
\hline
$2^k\cdot 3^m$&3&$3^2$&$3^3$&$3^4$&$3^5$&$3^6$&$3^7$&$3^8$&$3^9$\\
\hline
$2$  &r&r&r&r&r &-&-&-&-\\
$2^2$  &r&r&r,s&r,s&r&-&-&-&-\\
$2^3$ &r&r&r&r&-&-&-&-&-\\
$2^4$  &r&r&r,s&r&r,s&-&-&-&-\\
$2^5$  &r&r&r&r&-&-&-&-&*\\
$2^6$  &r&r&r,s&r&r&-&-&-&*\\
$2^7$  &r&r&r&-&-&-&-&*&*\\
$2^8$  &r&r&r,s&r&r,s&-&-&*&*\\
$2^9$  &r&r&r&-&-&-&*&*&*\\
$2^{10}$  &r&r&r,s&r&-&*&*&*&*\\
$2^{11}$  &r&r&r&-&-&*&*&*&*\\
$2^{12}$  &r&r&r&-&*&*&*&*&*\\
$2^{13}$  &r&r&-&*&*&*&*&*&*\\
$2^{14}$  &-&-&-&*&*&*&*&*&*\\
$2^{15}$  &-&-&*&*&*&*&*&*&*\\
$2^{16}$  &-&*&*&*&*&*&*&*&*\\
$2^{17}$  &-&*&*&*&*&*&*&*&*\\
\hline
\end{tabular}
\caption{}\label{App:Tab2by3}
\end{table}

\begin{table}[h]
\begin{tabular}{c|ccccccccc}
\hline
$2^k\cdot 3^m\cdot 5$&3&$3^2$&$3^3$&$3^4$&$3^5$&$3^6$&$3^7$&$3^8$&$3^9$\\
\hline
$2$  &r&r&r&r&r&-&-&-&-\\          
$2^2$  &r&r&r&r&r&-&-&-&-\\      
$2^3$ &r&r&r&-&-&-&-&-&*\\       
$2^4$  &r&r&r&r&-&-&-&-&*\\
$2^5$  &r&r&r&-&-&-&-&*&*\\
$2^6$  &r&r&r&r&r&-&*&*&*\\
$2^7$  &-&-&-&-&-&-&*&*&*\\
$2^8$  &r&r&r&r&-&*&*&*&*\\
$2^9$  &-&-&-&-&*&*&*&*&*\\
$2^{10}$  &r&r&r&-&*&*&*&*&*\\
$2^{11}$  &-&-&-&*&*&*&*&*&*\\
$2^{12}$  &-&-&-&*&*&*&*&*&*\\
$2^{13}$  &-&-&*&*&*&*&*&*&*\\
$2^{14}$  &-&*&*&*&*&*&*&*&*\\
\hline
$2^k\cdot 3^m\cdot 5^2$&3&$3^2$&$3^3$&$3^4$&$3^5$&$3^6$&$3^7$&$3^8$&$3^9$\\
\hline
$2$  &r&r&r&r&r&-&-&-&* \\
$2^2$  &r&r&r&r&r&-&-&*&*\\
$2^3$ &r&r&r&-&-&-&-&*&*\\
$2^4$  &r&r&r&r&-&-&*&*&*\\
$2^5$  &r&r&r&-&-&*&*&*&*\\
$2^6$  &r&r&r&r&r&*&*&*&*\\
$2^7$  &-&-&-&-&*&*&*&*&*\\
$2^8$  &r&r&r&r&*&*&*&*&*\\
$2^9$  &-&-&-&*&*&*&*&*&*\\
$2^{10}$  &r&r&r*&*&*&*&*&*&*\\
$2^{11}$  &-&-&*&*&*&*&*&*&*\\
$2^{12}$  &-&*&*&*&*&*&*&*&*\\
$2^{13}$  &*&*&*&*&*&*&*&*&*\\
$2^{14}$  &*&*&*&*&*&*&*&*&*\\
\hline
$2^k\cdot 3^m\cdot 5^3$&3&$3^2$&$3^3$&$3^4$&$3^5$&$3^6$&$3^7$&$3^8$&$3^9$\\
\hline
$2$  &r&r&r&-&- &-&-&*&*\\
$2^2$  &r&r&r&-&-&-&*&*&*\\
$2^3$ &-&-&-&-&-&*&*&*&*\\
$2^4$  &r&r&-&-&-&*&*&*&*\\
$2^5$  &-&-&-&-&-&*&*&*&*\\
$2^6$  &r&r&r&*&*&*&*&*&*\\
$2^7$  &-&-&-&*&*&*&*&*&*\\
$2^8$  &r&r&*&*&*&*&*&*&*\\
$2^9$  &-&-&-&*&*&*&*&*&*\\
$2^{10}$  &*&*&*&*&*&*&*&*&*\\
$2^{11}$  &*&*&*&*&*&*&*&*&*\\
$2^{12}$  &*&*&*&*&*&*&*&*&*\\
\hline
\end{tabular}
\caption{}\label{App:Tab2by3by5}
\end{table}

\begin{table}[h]
\begin{tabular}{c|ccccccccc}
\hline
$2^k\cdot 3^m\cdot 7$ &  $3$ & $3^{2}$ & $3^{3}$ & $3^{4}$ & $3^{5}$ & $3^{6}$ & $3^{7}$ & $3^{8}$ & $3^{9}$ \\
\hline
$2$  &r & r & r& r & - & - & - & - & -\\
$2^2$ &r & r & r& r & - & - & - & - & -\\
$2^3$ &r & r & r& r & - & - & - & - & *\\
$2^4$ &r & r & r& - & - & - & - & * & *\\
$2^5$ &- & - & -& - & - & - & - & * & *\\
$2^6$ &r & r & r & r & - & - & * & * & *\\
$2^{7}$ &- & - & -& - & - & * & * & * & *\\
$2^{8}$ &r & r & r & - & - & * & * & * & * \\
$2^{9}$ &- & - & -& - & * & * & * & * & * \\
$2^{10}$ & r & r & r & * & * & * & * & * & * \\
$2^{11}$ &- & - & - & * & * & * & * & * & * \\
$2^{12}$ &- & - & * & * & * & * & * & * & * \\
$2^{13}$ &- & - & * & * & * & * & * & * & * \\
$2^{14}$ &- & * & * & * & * & * & * & * & * \\
\hline
$2^k\cdot 3^m\cdot 7^2$ &$3$ & $3^{2}$ & $3^{3}$ & $3^{4}$ & $3^{5}$ & $3^{6}$ & $3^{7}$ & $3^{8}$ & $3^{9}$ \\
\hline
$2$  &r & r & r& r & - & - & - & * & *\\
$2^2$ &r & r & r& r & - & - & - & * & *\\
$2^3$ &r & r & r& - & - & - & * & * & *\\
$2^4$ &r & r & r& - & - & - & * & * & *\\
$2^5$ &- & - & -& - & - & * & * & * & *\\
$2^6$ &r & r & r & r & * & * & * & * & *\\
$2^{7}$ &- & - & -& - & * & * & * & * & *\\
$2^{8}$ &r & r & r & * & * & * & * & * & *\\
$2^{9}$ &- & - & *& * & * & * & * & * & *\\
$2^{10}$ & r & r & r *& * & * & * & * & * & *\\
$2^{11}$ &- & * & * & * & * & * & * & * & *\\
$2^{12}$ &* & * & * & * & * & * & * & * & *\\
$2^{13}$ &* & * & * & * & * & * & * & * & *\\
$2^{14}$ &* & * & * & * & * & * & * & * & *\\
\hline
$2^k\cdot 3^m\cdot 7^3$&3&$3^2$&$3^3$&$3^4$&$3^5$&$3^6$&$3^7$&$3^8$&$3^9$\\
\hline
$2$  &r&r&-&-&- &-&*&*&*\\
$2^2$  &r&r&-&-&-&-&*&*&*\\
$2^3$ &r&r&-&-&*&*&*&*&*\\
$2^4$  &-&-&-&-&*&*&*&*&*\\
$2^5$  &-&-&-&*&*&*&*&*&*\\
$2^6$  &r&r&*&*&*&*&*&*&*\\
$2^7$  &-&-&*&*&*&*&*&*&*\\
$2^8$  &-&*&*&*&*&*&*&*&*\\
$2^9$  &-&*&*&*&*&*&*&*&*\\
$2^{10}$  &*&*&*&*&*&*&*&*&*\\
$2^{11}$  &*&*&*&*&*&*&*&*&*\\
$2^{12}$  &*&*&*&*&*&*&*&*&*\\
\hline
$2^k\cdot 3^m\cdot 7^5$ &$3$ & $3^{2}$ & $3^{3}$ & $3^{4}$ & $3^{5}$ & $3^{6}$ & $3^{7}$ & $3^{8}$ & $3^{9}$\\
\hline
$2$  &-&-&*&*&* &*&*&*&*\\
$2^2$  &-&*&*&*&*&*&*&*&*\\
$2^3$ &-&*&*&*&*&*&*&*&*\\
$2^4$  &*&*&*&*&*&*&*&*&*\\
$2^5$  &*&*&*&*&*&*&*&*&*\\
$2^6$  &*&*&*&*&*&*&*&*&*\\
$2^7$  &*&*&*&*&*&*&*&*&*\\
$2^8$  &*&*&*&*&*&*&*&*&*\\
$2^9$  &*&*&*&*&*&*&*&*&*\\
$2^{10}$  &*&*&*&*&*&*&*&*&*\\
$2^{11}$  &*&*&*&*&*&*&*&*&*\\
$2^{12}$  &*&*&*&*&*&*&*&*&*\\
\hline
\end{tabular}
\caption{}\label{App:Tab2by3by7b}
\end{table}

\begin{table}[h]
\begin{tabular}{c|ccc||c|ccc}
\hline
$2^k\cdot 5^m\cdot 7$&5&$5^2$&$5^3$&$2^k\cdot 5^m\cdot 7^2$&5&$5^2$&$5^3$\\
\hline
$2$  &r&r&-&               $2$  &r&r&-\\
$2^2$  &r&r&-&         $2^2$  &-&r&-\\
$2^3$  &-&-&-&         $2^3$  &-&-&-\\
$2^4$  &-&-&-&         $2^4$  &-&-&-\\
$2^5$  & -&-&-&        $2^5$  & -&-&-\\
$2^6$  &r&r&-&         $2^6$  &r&r&-\\
$2^7$  &-&-&-&         $2^7$  &-&-&*\\
$2^8$  &-&-&-&         $2^8$  &-&-&*\\
$2^9$  &-&-&-&         $2^9$  &-&*&*\\
$2^{10}$  &-&-&*& $2^{10}$  &-&*&*\\
$2^{11}$  &-&-&*& $2^{11}$  &-&*&*\\
$2^{12}$  &-&*&*&$2^{12}$  &*&*&*\\
\hline
$2^k\cdot 5^m\cdot 7^3$&5&$5^2$&$5^3$&$2^k\cdot 5^m\cdot 7^5$&5&$5^2$&$5^3$  \\
\hline
$2$  &-&-&-&$2$  &-&*&*\\
$2^2$  &-&-&-&$2^2$  &-&*&*\\
$2^3$  &-&-&-&$2^3$  &*&*&*\\
$2^4$  &-&-&*&$2^4$  &*&*&*\\
$2^5$  & -&-&*&$2^5$  &*&*&*\\
$2^6$  &-&-&*&$2^6$  &*&*&*\\
$2^7$  &-&*&*&$2^7$  &*&*&*\\
$2^8$  &-&*&*&$2^8$  &*&*&*\\
$2^9$  &*&*&*&$2^9$  &*&*&*\\
$2^{10}$  &*&*&*&$2^{10}$  &*&*&*\\
$2^{11}$  &*&*&*&$2^{11}$  &*&*&*\\
$2^{12}$  &*&*&*&$2^{12}$  &*&*&*\\
\hline
\end{tabular}
\caption{}\label{App:Tab2by5by7}
\end{table}

%\newpage

\clearpage

\section{Acknowledgments}
The authors wish to thank Professor Kaplansky, in memoriam, for his encouragement to undertake this project, and Will Jagy for sharing frequent updates of the results of his ongoing computer search. The second author is grateful to the hospitality shown by Southern Illinois University Carbondale during her visits in which this project was completed.

\end{document}